\documentclass[reqno, 12pt]{amsart}
\usepackage[letterpaper,hmargin=1in,vmargin=1in]{geometry}
\usepackage{amsmath,amssymb,amsthm, hyperref}
\usepackage{verbatim,epsfig,graphics}
\usepackage{color}
\newtheorem{theorem}[equation]{Theorem}
\newtheorem{proposition}[equation]{Proposition}
\newtheorem{corollary}[equation]{Corollary}
\newtheorem{lemma}[equation]{Lemma}

\newtheorem{conj}[equation]{Conjecture}

\theoremstyle{remark}
\newtheorem{remark}[equation]{Remark}

\numberwithin{equation}{subsection}
\newcommand{\Vol}{\operatorname{Vol}}
\newcommand{\Tr}{\operatorname{Tr}}
\newcommand{\Id}{\operatorname{Id}}
\newcommand{\D}{\mathcal{D}}

\begin{document}
\title{Heat invariants of the Steklov problem}
\author{Iosif Polterovich}
\address{D\'e\-par\-te\-ment de math\'ematiques et de
sta\-tistique, Univer\-sit\'e de Mont\-r\'eal CP 6128 succ
Centre-Ville, Mont\-r\'eal QC  H3C 3J7, Canada.}
\email{iossif@dms.umontreal.ca}
\author{David A. Sher}
\address{Department of Mathematics and
Statistics, McGill University, 805 Sherbrooke Str. West,
Montr\'eal QC H3A 2K6, Ca\-na\-da.}
\address{Centre de Recherches Mathematiques, Univer\-sit\'e de Mont\-r\'eal CP 6128 succ
Centre-Ville, Mont\-r\'eal QC  H3C 3J7, Canada.}
\email{david.sher@mail.mcgill.ca}

\begin{abstract} We study the heat trace asymptotics associated with the Steklov eigenvalue problem on a Riemannian manifold with boundary. In particular, we
describe the structure of the Steklov heat invariants and compute the first few of them explicitly in terms of the scalar and mean curvatures. This is done by applying the Seeley calculus to the Dirichlet-to-Neumann operator, whose spectrum coincides with the Steklov eigenvalues. As an application, it is proved that a three--dimensional ball is uniquely defined by its Steklov spectrum among all Euclidean domains with smooth connected boundary.
\end{abstract}
\maketitle
\section{Introduction and main results}
\subsection{Steklov eigenvalue problem}
Let $\Omega$ be a smooth compact Riemannian manifold of dimension $n$ with smooth boundary $M=\partial \Omega$ of dimension $n-1$.  
Consider the Steklov eigenvalue problem on $\Omega$:
\begin{equation}
\label{stek}
\begin{cases}
  \Delta u=0& \mbox{ in } \Omega,\\
  \frac{\partial u}{\partial\nu}=\sigma \,u& \mbox{ on
  }M.
\end{cases}
\end{equation}
The spectrum of the Steklov problem is discrete and is given by a sequence of eigenvalues
$$
0=\sigma_0\leq \sigma_1 \le \sigma_2 \le \dots \nearrow \infty.
$$
The restrictions of the corresponding eigenfunctions to the boundary form an orthogonal basis in $L^2(M)$.

Geometric properties of  Steklov eigenvalues on Riemannian manifolds  have been actively investigated in the recent years, see \cite{CEG, FS1, FS2,  GP2,
Ja, KKP}.  In particular, various estimates on eigenvalues and their multiplicities have been obtained. 

\subsection{Steklov heat invariants}
\label{sec:heatinv}
The present paper focuses on geometric invariants associated with the Steklov problem. 
Steklov eigenvalues $\sigma_n$ may be also viewed as the eigenvalues of the Dirichlet-to-Neumann operator $\D:C^{\infty}(M) \to C^\infty(M)$,  given by $\D f=\partial_\nu(Hf)$, where $f \in C^{\infty}(M)$, $Hf \in C^\infty(\Omega)$ is its harmonic extension to $\Omega$ and $\partial_\nu$ denotes the outward normal derivative.
It is well known that $\D$ is a pseudodifferential operator of order one (see \cite[pp. 37-38]{Ta}).  
As follows from Weyl's law for the distribution of eigenvalues,  the dimension of $\Omega$ as well as the $(n-1)$--dimensional volume of the boundary $M=\partial \Omega$ can be determined from the Steklov spectrum.  Indeed,  by the results of \cite[Chapter 29]{ho},  the eigenvalue counting function of the Dirichlet-to-Neumann operator satisfies the asymptotics
\begin{equation}
\label{Weyllaw}
\#(\sigma_k \le \sigma) =  C_n \Vol(M) \sigma^{n-1} + O(\sigma^{n-2}),
\end{equation}
where $C_n$ is a constant depending on the dimension only (the explicit value of $C_n$ can be easily deduced from \eqref{wl2} and \eqref{a0eq}).
In order to obtain further geometric information, we consider the {\it heat trace asymptotics}. While such an approach is quite standard in spectral geometry (see \cite{g} and references therein), to our knowledge, it has not previously been systematically applied in the context of the Steklov problem.

As follows from  the results of \cite{DG, Ag, GrSe}, the trace of the associated heat kernel, $e^{-t\D}$, admits an asymptotic expansion 
\begin{equation}
\label{heatexp}
\sum_{i=0}^{\infty}e^{-t\sigma_i}=\Tr  e^{-t\D}=\int_M e^{-t\D}(x,x)\ dx\sim\sum_{k=0}^{\infty}a_kt^{-n+1+k}+\sum_{l=1}^{\infty}b_lt^l\log t.
\end{equation}
The coefficients $a_k$ and $b_l$ are called the Steklov heat invariants, and it follows from (\ref{heatexp}) that they are determined by the Steklov spectrum. The invariants $a_0,\ldots,a_{n-1}$, as well as $b_l$ for all $l$, are local, in the sense that they are integrals over $M$ of corresponding functions $a_k(x)$ and $b_l(x)$ which may be computed directly from the symbol of $\D$. However, $a_k$ is not local for $k\geq n$ \cite{g,gg}. One of the main goals of this paper is to investigate the local heat invariants $a_0(x)$, $\ldots$, $a_{n-1}(x)$. In particular, we discuss their general form and also compute $a_0$, $a_1$, and $a_2$ explicitly.

\subsection{Structure of the heat invariants}
We will prove that the Steklov heat invariants $a_k(x)$ must consist of explicit polynomials in the metric, its inverse, and their derivatives in both tangential directions along the boundary and directions normal to the boundary at $x$. For any term in such a polynomial, we let the \emph{(total) weight} be the total number of derivatives in that term and let the \emph{normal weight} be the total number of normal derivatives in that term; for example,  a term consisting of a tangential derivative of one metric element times a normal derivative of another metric element has weight 2 and normal weight 1. Note that by a Taylor series argument (see Section 2.1), all derivatives of the inverse metric at $x$ may be expressed as polynomials in the derivatives of the metric itself at $x$, with the same weight and the same normal weight as the original term. Moreover, we say that a polynomial in the metric, its inverse, and their derivatives has total weight $k$ if each term has total weight $k$. We also say that such a polynomial has \emph{normal parity} 1 if each term has odd normal weight, and normal parity 0 if each term has even normal weight. The following result is proved in section \ref{sec:structure}:
\begin{theorem}\label{structure} For each $k$ with $0\leq k\leq n-1$, the pointwise Steklov heat invariant $a_k(x)$ is a universal polynomial in the metric and its derivatives, of total weight $k$ and normal parity equal to $k$ mod 2.
\end{theorem}
\subsection{Explicit expressions}
We now give expressions for the first three Steklov heat invariants. Let $H_1$ be the mean curvature, and let $H_2$ be the \emph{second order mean curvature}, given in terms of the principal curvatures $\lambda_1$, $\ldots$, $\lambda_n$ by
\[H_2=\frac{1}{(n-1)(n-2)}\sum_{\alpha\neq\beta}\lambda_{\alpha}\lambda_{\beta}.\]
Further, let $R_{\Omega}$ be the scalar curvature of the domain $\Omega$ and $R_M$ be the scalar curvature of the boundary $M$. Finally, let 
$$
V_n=\Vol(\mathbb{S}^{n-2})=\frac{2\pi^{\frac{n-1}{2}}}{\Gamma(\frac{n-1}{2})}.
$$
\begin{theorem}\label{explicit} The first three pointwise heat invariants $a_k(x)$ of the Steklov spectrum of an $n$--dimensional Riemannian manifold $\Omega$ with boundary $M=\partial \Omega$ are given by the formulas
\begin{equation}\label{zeroinvar}a_0(x)=(2\pi)^{-n+1}V_n\Gamma(n-1), \,\,\,\, n\ge 1; \end{equation}
\smallskip
\begin{equation}\label{firstinvar}a_1(x)=\frac{V_n(n-2)\Gamma(n-1)}{2(2\pi)^{n-1}}H_1, \,\,\,\, n \ge 2; \end{equation}
\smallskip
\begin{multline}
\label{seconinvar}
a_{2}(x)=\frac{V_n\Gamma(n-2)}{8(2\pi)^{n-1}}\left[\frac{(n-1)(n-2)(n^2-n-4)}{n+1}H_1^2
\,-\frac{n(n-3)(n-2)}{n+1}H_2\,\right.\\ \left.+\frac{n-2}{n-1}R_{\Omega}-\frac{n-4}{3(n-1)}R_M\right],  \,\,\,\, n \ge 3.
\end{multline}
\end{theorem}
\begin{remark} Edward and Wu have shown in \cite{ew} that $a_1=0$ for any simply connected domain $\Omega$ in $\mathbb R^2$; this agrees with our result, which indicates that in fact the pointwise heat invariant $a_1(x)=0$ whenever $n=2$.
\end{remark}
Theorem \ref{explicit} is proved in sections \ref{41} and \ref{42}. 
The expression for $a_2(x)$ simplifies significantly if $\Omega$ has constant sectional curvature $K$. 
In this case, $R_{\Omega}=n(n-1)K$ and $R_M=(n-1)(n-2)(K+H_2)$ (see \cite{alm}). We therefore have
\begin{corollary} When $\Omega$ has constant sectional curvature $K$,
\begin{multline}a_{2}(x)=\frac{V_n\Gamma(n-2)}{8(2\pi)^{n-1}}\left[\frac{(n-1)(n-2)(n^2-n-4)}{n+1}H_1^2-\frac{4(n^2-3n-1)}{3(n^2-1)}R_M\,\right.\\ \left.+\frac{2n(n-1)(n-2)}{n+1}K\right].\end{multline}
In particular, when $n=3$ and $\Omega \subset \mathbb R^3$,
\begin{equation}
\label{3d}
a_2(x)=\frac{1}{16\pi}\left(H_1^2+\frac{R_M}{6}\right).
\end{equation}
\end{corollary}
Apart from Euclidean domains ($K=0$), examples covered by the corollary include domains in spheres ($K>0$) and in hyperbolic spaces ($K<0$).
\subsection{Applications to spectral rigidity} The coefficients $a_k$, $k=0,\dots,n-1$, in the heat trace expansion \eqref{heatexp} are given by
$$
a_k=\int_M a_k(x).
$$
As was mentioned in section \ref{sec:heatinv}, it follows from Weyl's law that the volume of $M$ is determined by the Steklov spectrum. Formula \eqref{zeroinvar} 
provides an alternative proof of this result, as 
\begin{equation}
\label{a0eq}
a_0=\frac{\Gamma(n-1)}{2^{n-2} \pi^{\frac{n-1}{2}}\Gamma(\frac{n-1}{2})} \Vol(M).
\end{equation}
For $n=2$, the first heat invariant $a_1=0$, but when $n\geq 3$, $a_1$ is a nonzero multiple of the total mean curvature. As a consequence, we obtain
\begin{proposition} When $n\geq 3$, the total mean curvature $\int_M H_1$ of the boundary $M=\partial \Omega$ is determined by the Steklov spectrum.
\end{proposition}

In sections \ref{51} and \ref{52}, we prove the following spectral rigidity result: 
\begin{theorem}
\label{corollary:main}
 Let $\Omega\subset\mathbb R^3$ be a domain with connected and smooth boundary $M$. Suppose its Steklov spectrum is equal to that of $B_\rho$, a ball of radius $\rho$. Then $\Omega=B_\rho$. \end{theorem}
The proof of the theorem is split into two parts. First,  we use eigenvalue and multiplicity asymptotics to show that $M$ is  simply connected; this step is based on an argument of Zelditch \cite{z}. We then use the heat invariants computed above to show that $M$ is a sphere.
\begin{remark}
\label{Weinstock}
An analogue of Theorem \ref{corollary:main} in two dimensions is well--known. Indeed, it was shown by Weinstock (\cite{We}, see also \cite{GP1}) that among simply connected planar domains, of a given perimeter,  the first Steklov eigenvalue attains its unique maximum on a disk. Since perimeter is spectrally determined, this implies that a disk is uniquely determined by its Steklov spectrum among all simply connected planar domains. An alternative proof of this result could be found in \cite{e}.
\end{remark}

It is unknown whether a ball maximizes the first Steklov eigenvalue in higher dimensions among all domains of fixed volume of the boundary; therefore, the above argument can not be used in higher dimensions.  There exist generalizations of Weinstock's inequality due to  \cite{Brock} and \cite{BS},
but they involve $\Vol(\Omega)$, and it is not known whether it is an invariant of the Steklov spectrum.

We conclude this section with the following conjecture, motivated by 
Theorem \ref{corollary:main} and Remark \ref{Weinstock}.
\begin{conj}
A ball in $\mathbb{R}^n$ is uniquely determined by its Steklov spectrum among all $n$-dimensional Euclidean domains.
\end{conj}
Note that the analogue of this conjecture holds for Dirichlet and Neumann eigenvalue problems. Indeed, 
the volumes of both the domain and its boundary  can be determined from either the Dirichlet or the Neumann spectrum (see \cite{g}). At the same time, the ball is a unique Euclidean domain for which the classical isoperimetric inequality turns into an equality. Alternatively, one could use the Faber--Krahn (respectively, Szeg\H{o}--Weinberger) inequality, stating that the ball is a unique minimizer (respectively, maximizer) of the first Dirichlet (respectively, first nonzero Neumann) eigenvalue among all Euclidean domains of given volume (see \cite{He}).

There are other situations in which similar rigidity theorems can be proved using heat invariants. For example, the round sphere is known to be determined by the spectrum of the Laplace-Beltrami operator in dimensions $n\leq 6$ \cite{tan}. Hassell and Zworski also use heat invariants to prove resonant rigidity for $\mathbb{S}^2$ in the setting of obstacle scattering in $\mathbb R^3$ \cite{HZ}.

\subsection{Plan of the paper}  In sections \ref{sec:coord} and \ref{sec:symbol} we compute the symbol of the Dirichlet-to-Neumann operator in boundary normal coordinates, following \cite{lu}.  Combining these results with the Seeley calculus, we obtain local expressions for the first $n-1$ heat invariants in 
section~\ref{sec:Seeley}.  This allows us to prove Theorem \ref{structure} in section \ref{sec:structure}. Some applications and extensions of this result are presented in section \ref{sec:applications}.  In particular, the  heat invariants of product--type manifolds are described in Corollary \ref{producttype}. Sections \ref{41}
and \ref{42} are devoted to the proof of Theorem \ref{explicit}.  In section \ref{sec:example} we verify these results by a direct calculation of the first few  heat invariants of Euclidean balls in dimensions three and four. 
Finding an explicit expression for the heat invariant $a_2(x)$ is the most technically involved part of the paper. The proofs of auxiliary lemmas used in this calculation are presented in  sections \ref{pcompute} and \ref{pscal}. Some integrals  appearing in the computations are given in section \ref{auxint}. Finally, in sections \ref{51} and \ref{52}, we prove Theorem \ref{corollary:main}.

\subsection*{Acknowledgments} The authors would like to thank P. Hislop and P. Perry for useful discussions, and the anonymous referee for helpful remarks.  Research of I.P. was supported in part by NSERC, FQRNT and the Canada Research Chairs Program. 
Research of D.S. was supported in part by the CRM-ISM postdoctoral fellowship.

\section{Seeley calculus for the Dirichlet-to-Neumann operator}
\subsection{Specialized coordinates}
\label{sec:coord}
To compute the first few terms in the symbol of $\mathcal D$, we follow the work of Lee and Uhlmann \cite{lu}. Since the heat invariants $a_k(x)$ are local, we are free to choose coordinates which are well-adapted to our analysis near a fixed point $P$ and then compute $a_k(P)$. We use Riemannian normal coordinates centered at $P$, with any choice of initial orthonormal frame, to write the metric on $M$ in coordinates as
\[\sum_{\alpha,\beta=1}^{n-1}g_{\alpha\beta}(x)dx^{\alpha}dx^{\beta}.\]
Here $P$ corresponds to the origin and $g_{\alpha\beta}(x)=\delta_{\alpha\beta}+\mathcal O(x^2)$, where $\delta_{\alpha\beta}$ is the Kronecker delta. Let the inverse metric be $g^{\alpha\beta}$. Notice that all first derivatives of the metric and its inverse vanish at $P$.

Following \cite{lu}, we now extend these coordinates to 'boundary normal coordinates' in a neighborhood of $M\subset\Omega$. For each $x\in M$ near $P$, let $x^n$ be the parameter along the geodesic starting at $x$ with initial direction given by the inward-pointing boundary normal. Then the coordinates $(x^{\alpha}|_{\alpha=1}^{n-1},x^n)$, which we write $x'=(x,x^n)$ (note that the notation differs slightly from that in \cite{lu}), are local coordinates on $\Omega$ in a neighborhood of $P$. Moreover, the definition of $x^n$ does not depend on the initial orthonormal frame. In these coordinates, the metric is precisely
\[g=\sum_{\alpha,\beta=1}^{n-1}g_{\alpha\beta}(x')dx^{\alpha}dx^{\beta}+(dx^n)^2.\] As in \cite{lu}, we use the Greek indices for coordinates along $M$ and Roman indices for coordinates in $\Omega$. Let the dual coordinates to $x^{\alpha}$ be $\xi^{\alpha}$, and let the squared volume element $\delta(x')=\det(g_{ij}(x'))=\det(g_{\alpha\beta}(x'))$.

Using these coordinates, we still have freedom to choose our initial orthonormal frame on $M$; in particular, we may choose a frame that diagonalizes the second fundamental form. First fix any orthonormal frame $X_1,\ldots, X_{n-1}$, and let the local normal be $X_n$ (which is independent of the choice of frame). Then, letting $\Gamma_{ij}^k$ be the Christoffel symbols for $g$, we may write the linear operator associated to the second fundamental form as
\[S(X_{\beta})=\sum_{\alpha=1}^{n-1}-\Gamma_{\beta n}^{\alpha}(P)X_{\alpha}.\]
Using the formula for the Christoffel symbols (see, e.g. \cite{dc}) together with the fact that $g(P)=\Id$, we have, writing $\partial_k g_{ij}$ as $g_{ij,k}$:
\[-\Gamma^{\alpha}_{\beta n}=-\frac{1}{2}(g_{n\alpha,\beta}+g_{\alpha\beta,n}-g_{\beta n,\alpha}).\]
However, since $g_{n\alpha}=0$ in a neighborhood of $P$ in $\Omega$, we have
\[-\Gamma^{\alpha}_{\beta n}=-\frac{1}{2}g_{\alpha\beta,n}.\]
So the $(\alpha,\beta)$ entry of the matrix of the second fundamental form at $P$ is just $-\frac{1}{2}g_{\alpha\beta,n}(P)$.
This is a symmetric matrix and hence orthogonally diagonalizable, so we may choose an initial orthonormal frame which diagonalizes the second fundamental form at $P$. In these coordinates, $g_{\alpha\beta,n}(P)=0$ whenever $\alpha\neq\beta$, and $g_{\alpha\alpha,n}(P)=-2\lambda_{\alpha}$, where $\lambda_{\alpha}$ is the corresponding principal curvature at $P$.

It will also prove helpful to have similar expressions for the inverse metric at $P$. Viewing the metric $g(x')$ as a function of $x^n$ with values in symmetric 2-tensors on $M$, we Taylor expand around $x^n=0$:
\[g(x')=\Id +A_1x^n+A_2(x^n)^2+\mathcal O((x^n)^3).\]
The inverse metric then must have the expansion
\[g^{-1}(x')=\Id-A_1x^n+(A_1^2-A_2)(x^n)^2+\mathcal O((x^n)^3).\]
We conclude that $g^{\alpha\beta,n}=-g_{\alpha\beta,n}$, and in particular that $g^{\alpha\beta,n}(P)=2\lambda_{\alpha}\delta_{\alpha\beta}$. Moreover, from the second order terms, we observe that \begin{equation}\label{secondorderexp}\sum_{\alpha}g^{\alpha\alpha,nn}(P)=\sum_{\alpha}8\lambda_{\alpha}^2-\sum_{\alpha}g_{\alpha\alpha,nn}(P).\end{equation}

\subsection{The symbol of the Dirichlet-Neumann operator}
\label{sec:symbol}
Letting $D_{x^j}=-i\partial_{x^j}$, we write the Laplacian on $\Omega$ as
\[D_{x^n}^2+iE(x')D_{x^n}+Q(x',D_{x}),\]
where
\[E(x')=-\frac{1}{2}\sum_{\alpha,\beta}g^{\alpha\beta}(x')\partial_{x^n}g_{\alpha\beta}(x'),\]
\[Q(x',D_{x})=\sum_{\alpha,\beta}g^{\alpha\beta}(x')D_{x^{\alpha}}D_{x^{\beta}}
-i\sum_{\alpha,\beta}(\frac{1}{2}g^{\alpha\beta}(x')\partial_{x^{\alpha}}\log\delta(x')
+\partial_{x^{\alpha}}g^{\alpha\beta}(x'))D_{x^{\beta}}.\]
In particular, we can read off the full symbol of $Q$, which we write $q_2(x',\xi)+q_1(x',\xi)$.

By observing that the Dirichlet-to-Neumann operator $\D$ satisfies a Riccati-type equation, Lee and Uhlmann compute its full symbol \cite{lu}. First, define the symbol of a pseudodifferential operator with parameter $x_n$, which we call $\hat{\D}$, recursively by the following formulas, corresponding to (1.7), (1.8), and (1.9) in \cite{lu}:
\begin{equation}\label{r1}\hat r_1=-\sqrt{q_2},\end{equation}
\begin{equation}\label{r0}\hat r_0=\frac{1}{2\sqrt{q_2}}\left(\sum_{\alpha}d_{\xi^\alpha}\sqrt{q_2}D_{x^\alpha}\sqrt{q_2}
-q_1-\partial_{x^n}\sqrt{q_2}+E\sqrt{q_2}\right),\end{equation}
and for $m\geq 0$,
\begin{equation}\label{recursive}\hat r_{-m-1}=\frac{1}{2\sqrt{q_2}}\left(\sum_{\substack{-m\leq j\leq 1\\ -m\leq k\leq 1\\ |K|=j+k+m}}\frac{1}{K!}\partial_{\xi}^K(\hat r_j)D_{x'}^K(\hat r_k)+\partial_{x^n}\hat r_{-m}-E\hat r_{-m}\right).\end{equation}
Let $\tilde{\D}$ be the restriction of $\hat{\D}$ to $x_n=0$. Then as in \cite{lu} the Dirichlet-to-Neumann operator $\D$ is equal, modulo smoothing operators, to $-\tilde{\D}$; the sign is chosen so that $\D$ has positive principal symbol. Write the symbol of $\D$ as $r_1+r_0+r_{-1}+b_{-2}$, where $b_{-2}$ has order $-2$; expressions for $r_1$, $r_0$, and $r_{-1}$ may be computed from (\ref{r1}), (\ref{r0}), and (\ref{recursive}).

\subsection{Seeley calculus and the heat invariants}
\label{sec:Seeley}
One may now apply the work of Seeley \cite{s} to compute the local expressions for the first $n-1$ heat invariants of $\D$. We follow the exposition of Gilkey and Grubb \cite{gg}. Namely, let $S(\lambda)$ be a parametrix for $\D-\lambda$: that is, a pseudodifferential operator of order $-1$ with parameter $\lambda$ for which $(\D-\lambda)S(\lambda)$ and $S(\lambda)(\D-\lambda)$ are both pseudodifferential operators of order $-\infty$. Such a pseudodifferential operator must have symbol $s_{-1}(x,\xi,\lambda)+s_{-2}(x,\xi,\lambda)+\ldots$ given by:
\begin{equation}\label{sminusone}s_{-1}(x,\xi,\lambda)=\frac{1}{r_1-\lambda},\end{equation}
\begin{equation}\label{recursives}s_{-1-m}(x,\xi,\lambda)=-\frac{1}{r_1-\lambda}\left(\sum_{\substack{-m\leq j\leq -1\\ -m\leq k\leq 1\\ |K|=m+k+j\geq 0}}\frac{(-i)^{|K|}}{K!}\partial_{\xi}^Kr_k\partial_{x}^Ks_j\right).\end{equation}
For later reference, we write out the formulas for $s_{-2}$ and $s_{-3}$:
\begin{equation}\label{sminustwo}s_{-2}=-r_0(r_1-\lambda)^{-2}-i(\partial_{\xi}r_1\cdot\partial_{x}r_1)(r_1-\lambda)^{-3};\end{equation}
\begin{equation}s_{-3}=-(r_1-\lambda)^{-1}\left[r_0s_{-2}+r_{-1}s_{-1}-i(\partial_{\xi}r_1\cdot\partial_{x}s_{-2}-\partial_{\xi}r_0\cdot\partial_{x}s_{-1})\label{sminusthree}-\sum_{|K|=2}\frac{1}{K!}\partial_{\xi}^Kr_1\partial_{x}^Ks_{-1}\right].\end{equation}
The heat invariants may now be computed by using the functional calculus; see \cite{s} for the details. Letting $\Gamma$ be a contour around the positive real axis and letting $T^*_xM$ be the cotangent space at $x$, we have that for $0\leq k\leq n-1$,
\begin{equation}\label{heatinv}a_k(x)=\frac{i}{(2\pi)^{n}}\int_{T^*_{x}M}\int_{\Gamma}e^{-\lambda}s_{-1-k}(x,\xi,\lambda)\ d\lambda\ d\xi.\end{equation}

\section{General form of the heat invariants}
\subsection{Structure theorem}
\label{sec:structure}
Here we prove Theorem \ref{structure}. The proof proceeds by first analyzing the inductive formulas for the symbol of $\D$ and then passing to the inductive formulas for the heat invariants. In these formulas, we repeatedly encounter expressions of the form
\begin{equation}\label{termform} |\xi|^{-k}p(x',\xi),
\end{equation}
where $k\in\mathbb Z$, $|\xi|^2$ is shorthand for $g^{\alpha\beta}\xi^{\alpha}\xi^{\beta}$, and $p(x',\xi)$ is a polynomial in $\xi$ with coefficients equal to polynomials in the metric, its inverse, and their derivatives. We say that an expression of the form (\ref{termform}) has total weight $w$ if each term in $p(x',\xi)$ has total weight $w$. Additionally, we say that such an expression has total parity equal to 1 if each term in $p(x',\xi)$ has odd total weight, or equal to 0 if each term has even total weight; note that the total parity is not defined if there are terms of both odd and even total weight in $p(x',\xi)$. Analogous definitions may be formulated for normal weight and normal parity. Finally, we define the $\xi$-parity of (\ref{termform}) to be 1 if $p(x',\xi)$ is odd in $\xi$ and 0 if $p(x',\xi)$ is even in $\xi$. 

Note that each of these definitions is independent of a change of form in (\ref{termform}) - for instance, increasing $k$ by 2 and multiplying $p(x',\xi)$ by $|\xi|^2$ changes nothing.
Furthermore: suppose $\hat p(x',\xi)=|\xi|^{-k}p(x',\xi)$ has well-defined total weight, normal parity, and $\xi$-parity. Then it is an easy calculation to see that each of $\partial_{x^n}\hat p$, $\partial_{x^{\alpha}}\hat p$, and $\partial_{\xi^{\alpha}}\hat p$ may also be written in the form (\ref{termform}), and each has well-defined total weight, normal parity, and $\xi$-parity. All $x'$-derivatives increase the total weight by 1, the $x^n$-derivative also flips the normal parity, and any $\xi$-derivative flips the $\xi$-parity.

\begin{lemma} For each $m\geq -1$, $\hat r_{-m}(x')$ is a sum of terms of the form (\ref{termform}), each of which has total weight $m+1$. Moreover, the total parity, normal parity, and $\xi$-parity of each term are all well-defined, and their sum is equal to zero mod 2.
\end{lemma}
Note that since $\D=-\hat{\D}|_{x^n=0}$, an identical statement obviously holds for $r_j(x,\xi)$.
\begin{proof}
The proof is by induction on $m$; the lemma may be verified by explicit computation for $\hat r_1$ and $\hat r_0$. Now assume the inductive hypothesis holds for $\hat r_{-m}$, $\ldots$, $\hat r_1$, and use the equation (\ref{recursive}) to compute $\hat r_{-m-1}$ term by term. Multiplication by $E$ increases both the total and normal weight by 1 while leaving the $\xi$-parity unchanged, and differentiation in $x^n$ does the same. Therefore $|2 \xi|^{-1}(-E\hat r_{-m}+\partial_{x^n}\hat r_{-m})$ has the properties we claim. The only remaining term in the expression for $\hat r_{-m-1}$ is
\begin{equation}\label{hereisasum}\frac{1}{2|\xi|}\sum_{\substack{-m\leq j\leq 1\\ -m\leq k\leq 1\\ |K|=j+k+m}}\frac{1}{K!}\partial_{\xi}^K(\hat r_j)D_{x}^K(\hat r_k).\end{equation}
The total weight of each term is $(-j+1)+(-k+1)+K=m+2$, as required. Moreover, by the inductive hypothesis, the sum of the total weights, normal parities, and $\xi$-parities of each term in $\hat r_j$ and $\hat r_k$ is equal to zero mod 2 for all $j$ and $k$. Therefore, the sum of the total weights, normal parities, and $\xi$-parities of each term in (\ref{hereisasum}) is $K+K=0$ mod 2. This completes the proof of the lemma. \end{proof} 

The next step is to pass to the parametrix $S(\lambda)$. In this analysis, expressions of the form
\begin{equation}\label{termformtwo}(|\xi|-\lambda)^{-l}|\xi|^{-k}p(x,\xi)\end{equation} arise frequently; we define weights and parities of such expressions analogously to those for expressions of the form (\ref{termform}). Even though these expressions are restricted to $x^n=0$, they may involve normal derivatives of the metric, so normal weight and parity still make sense. A similar lemma holds for the $s_j$:

\begin{lemma} For each $m\geq -1$, $s_{-m}(x,\xi,\lambda)$ may be written as a sum of terms of the form (\ref{termform}); each summand has total weight $m-1$. Moreover, the total parity, normal parity, and $\xi$-parity of each summand are all well-defined, and their sum is equal to zero mod 2. \end{lemma}
\begin{proof}
The proof is extremely similar to the previous lemma, and again proceeds by induction on $m$. It is easy to check $s_{-1}$ explicitly. Now assume the inductive hypothesis.  Using the previous lemma and the inductive hypothesis, the total weight of each term in the sum (\ref{recursives}) for $s_{-m-1}$ is $(-k+1)+(-j-1)+K=m$, as required. Moreover, the sum of the total weights, normal parities, and $\xi$-parities of each term in $r_k$ and $s_j$ is equal to zero mod 2.  As before, we conclude that the sum of the total weights, normal parities, and $\xi$-parities of each term in (\ref{recursives}) is $K+K=0$ mod 2. This completes the proof. \end{proof}

We now pass from the parametrix $S(\lambda)$ to the heat invariants $a_k(x)$ themselves by using formula (\ref{heatinv}): fixing $x$, integrate in $\lambda$ and in $\xi$. Each individual term in $s_{-1-k}$ has the form (\ref{termformtwo}). The contour integral in (\ref{heatinv}) may be explicitly computed using Lemma \ref{contour}; the result for each term is $e^{-|\xi|}$ times an expression of the form (\ref{termform}) which has the same total weight, normal parity and $\xi$-parity as the terms in $s_{-1-k}$. Now perform the $\xi$ integral; this integral will vanish identically for all terms with $\xi$-parity 1, and so each term which survives is a polynomial in the metric, its inverse, and their derivatives at $x$, with total weight and normal parity summing to zero mod 2. Since the total weight is $k$, each term in $a_k(x)$ has normal parity equal to $k$ mod 2. This completes the proof of Theorem \ref{structure}.

\subsection{Generalizations and applications} 
\label{sec:applications}
Recall that for the Laplacian on a closed Riemannian manifold, the local heat invariants are polynomials in the curvature tensor and its covariant derivatives \cite{g}. We may re-write Theorem \ref{structure} in this form as well. Let $Riem_{\Omega}$ be the curvature tensor for $\Omega$ and let $\mathcal H$ be the second fundamental form of $M\subset\Omega$; also let $\nabla_{\Omega}$ and $\nabla_M$ be the covariant derivatives on $\Omega$ and $M$ respectively. Then

\begin{theorem}\label{refined} For $1\leq k\leq n-1$, the local Steklov heat invariant $a_k(x)$ may be written as a universal polynomial in the entries of the tensors $\nabla^{j}_{\Omega}Riem_{\Omega}$, $0\leq j\leq k-2$, and $\nabla^j_M\mathcal H$, $0\leq j\leq k-1$.
\end{theorem}
\begin{proof} The statement is an immediate consequence of Theorem 1.1.3 in \cite{g}. The restrictions on $j$ follow from the weights in Theorem \ref{structure}. \end{proof}

An interesting special case is when the embedding $M\subset\Omega$ is \emph{product type}: suppose that there exists a $\delta>0$ and a tubular neighborhood $U$ of $M$ in $\Omega$ such that $U$ is isometric to $M\times [-\delta,\delta]_{x^n}$ with the product metric. In this situation, the second fundamental form $H$ is identically zero, as are all derivatives of the metric and its inverse in the $x^n$ direction. We claim
\begin{corollary}\label{producttype} If the embedding $M\subset\Omega$ is product type, the heat invariants $a_k(x)$ vanish for any odd $k$ between $1$ and $n-1$. On the other hand, for even $k$ between $1$ and $n-1$, the heat invariants are polynomials in the entries of the boundary curvature tensor $Riem_M$ and its covariant derivatives $\nabla_M^{j}Riem_M$, $0\leq j\leq k-2$.
\end{corollary}
\begin{proof} The vanishing of the odd heat invariants follows immediately from Theorem \ref{structure}; each term in the universal polynomial for an odd heat invariant contains an odd number of normal derivatives of the metric, and hence contains at least one such normal derivative. Since the embedding is product type, each such term must be identically zero, and therefore the whole invariant is zero.

As for the even heat invariants, Theorem \ref{refined} indicates that they are universal polynomials in $Riem_{\Omega}$ and its $\Omega$-covariant derivatives. Since all normal derivatives vanish, we may rewrite all of the $\Omega$-covariant derivatives of $Riem_{\Omega}$ in terms of $M$-covariant derivatives of $Riem_M$. Finally, examining the weights in Theorem \ref{structure}, the corollary follows immediately.
\end{proof}

Note that in the product-type case, it is known that the Dirichlet-to-Neumann operator is a square root of the boundary Laplacian, in the sense that its square is equal, modulo infinitely smoothing operators, to the boundary Laplacian \cite{l}.

\section{Computation of heat invariants}

In this section, we compute and analyze the first few heat invariants, proving Theorem \ref{explicit}. Throughout, we fix a point $P$ on $M$ and use the $P$-centered coordinates from Section 2.1 to compute $a_k(P)$. Recall from Section 2.1 that we may choose an initial orthonormal frame on $M$ which diagonalizes the second fundamental form at $P$; we use this frame in all calculations that follow.

\subsection{Computations of $a_0$ and $a_1$} 
\label{41}
We now evaluate the expressions (\ref{r1}), (\ref{r0}), (\ref{sminusone}), and (\ref{sminustwo}) at $P$ using our local coordinates. These expressions, written out, are
\[\hat r_1=-\sqrt{q_2}=-\sqrt{g^{\alpha\beta}\xi^{\alpha}\xi^{\beta}};\]
\[\hat r_0=\frac{1}{8q_2^{3/2}}(\sum_{\gamma}-i\partial_{\xi^{\gamma}}q_2\partial_{x^{\gamma}}q_2)-\frac{q_1}{2\sqrt{q_2}}-\frac{1}{4q_2}\partial_{x^n}q_2+\frac{E}{2}\]
\[=\frac{-i}{8q_2^3}\sum_{\gamma}(\sum_{\alpha,\beta} g^{\alpha\beta}(\delta^{\gamma\alpha}\xi^{\beta}+\delta^{\gamma\beta}\xi^{\alpha}))(\sum_{\alpha,\beta}g^{\alpha\beta,\gamma}\xi^{\alpha}\xi^{\beta})-\frac{1}{4q_2^2}\sum_{\alpha,\beta}g^{\alpha\beta,n}\xi^{\alpha}\xi^{\beta}\]\begin{equation}\label{writtenout}+\frac{i}{2q_2}\sum_{\alpha,\beta}(\frac{1}{2}g^{\alpha\beta}(x')\partial_{x^{\alpha}}\log\delta(x')
+\partial_{x^{\alpha}}g^{\alpha\beta}(x'))\xi^{\beta}
-\frac{1}{4}\sum_{\alpha,\beta}g^{\alpha\beta}g_{\alpha\beta,n}.\end{equation}
Since we are using Riemannian normal coordinates at $P$, any first derivative in $x^{\alpha}$ of the metric or its inverse, for $1\leq\alpha\leq n-1$, vanishes at $P$. Using (\ref{writtenout}) and the fact that $\D$ is $-1$ times the restriction of $\hat\D$ to $x^n=0$, we obtain that $r_1(P,\xi)=|\xi|$ and
\begin{equation}\label{rzeroatp}r_0(P,\xi)= \frac{1}{2|\xi|^2}\sum_{\alpha}\lambda_{\alpha}(\xi^{\alpha})^2-\frac{1}{2}\sum_{\alpha}\lambda_{\alpha}.\end{equation}
Note that the subprincipal symbol of $\mathcal D$, given in (\ref{rzeroatp}), has previously been computed by Taylor (Ch. 12, Proposition C1 in \cite{Ta}).  We may also simplify (\ref{sminusone}) and (\ref{sminustwo}); note in particular that any tangential first derivative of $r_1$ vanishes at $P$. Also remember that $g^{\alpha\beta}(P)=\delta^{\alpha\beta}$, $g_{\alpha\beta,n}(P)=-2\lambda_{\alpha}\delta_{\alpha\beta}$, and $g^{\alpha\beta,n}(P)=2\lambda_{\alpha}\delta_{\alpha\beta}$. The results are
\[s_{-1}(P,\xi,\lambda)=(|\xi|-\lambda)^{-1};\]\[s_{-2}(P,\xi,\lambda)=-(|\xi|-\lambda)^{-2}(\frac{1}{2|\xi|^2}\sum_{\alpha}\lambda_{\alpha}(\xi^{\alpha})^2-\frac{1}{2}\sum_{\alpha}\lambda_{\alpha}).\]

Now plug these expressions into (\ref{heatinv}). The contour integrals may be computed explicitly:
\begin{lemma}\label{contour} For any $k\geq 1$ and any $a\in\mathbb R^+$,
\[\int_{\mathcal C}\frac{e^{-\lambda}}{(a-\lambda)^k}\ d\lambda=-2\pi i\frac{e^{-a}}{(k-1)!}.\]
\end{lemma}
\begin{proof} The proof is a simple computation with calculus of residues.
\end{proof}
Using the contour integrals, then using the integrals in section \ref{auxint} to evaluate the $\xi$ integrals, we have, where $V_n=\Vol(\mathbb{S}^{n-2})$,
\begin{equation}\label{a0}a_0(P)=(2\pi)^{-n+1}\int_{\mathbb R^{n-1}}e^{-|\xi|}\ d\xi=(2\pi)^{-n+1}V_n\Gamma(n-1);\end{equation}
\[a_1(P)=(2\pi)^{-n+1}\sum_{\alpha}\frac{\lambda_{\alpha}}{2}\int_{\mathbb R^{n-1}}e^{-|\xi|}(1-\frac{(\xi^{\alpha})^2}{|\xi|^2})\ d\xi=\frac{V_n(n-2)\Gamma(n-1)}{2(n-1)(2\pi)^{n-1}}\sum_{\alpha}\lambda_{\alpha}\]\begin{equation}\label{a1}=\frac{V_n(n-2)\Gamma(n-1)}{2(2\pi)^{n-1}}H_1.\end{equation} This completes the proof of the expressions for $a_0$ and $a_1$ in Theorem \ref{explicit}.

\subsection{Computation of $a_2$} 
\label{42}
The computation of $a_2$ is somewhat more involved; this is typical of heat invariant calculations, which tend to increase dramatically in complexity as one goes farther out in the expansion. As a starting point, we write out the expression for $s_{-3}$ given in (\ref{sminusthree}):

\[s_{-3}(P,\xi,\lambda)=(r_0(P))^2(|\xi|-\lambda)^{-3}-(r_{-1}(P))(|\xi|-\lambda)^{-2}-i(\partial_{\xi}r_1\cdot\partial_{x}r_0)(|\xi|-\lambda)^{-3}\]\[+(\partial_{\xi}r_1\cdot\partial_{x}(\partial_{\xi}r_1\cdot\partial_{x}r_1))(|\xi|-\lambda)^{-4}+(r_1-\lambda)^{-1}(\sum_{|K|=2}\frac{1}{K!}(\partial_{\xi}^Kr_1)(P)(\partial_{x}^Ks_{-1})(P)).\]
Simplifying, using the fact that first derivatives in $x$ of $r_1$, and therefore also first derivatives in $x$ of $s_{-1}$, vanish at $P$:
\[s_{-3}(P,\xi,\lambda)=(r_0(P))^2(|\xi|-\lambda)^{-3}-(r_{-1}(P))(|\xi|-\lambda)^{-2}-i(\partial_{\xi}r_1\cdot\partial_{x}r_0)(|\xi|-\lambda)^{-3}\]\[+(\sum_{\gamma,\epsilon}(\partial_{\xi^\epsilon}r_1\partial_{\xi^{\gamma}}r_1\partial_{x^{\gamma}}\partial_{x^{\epsilon}}r_1))(|\xi|-\lambda)^{-4}-(\sum_{|K|=2}\frac{1}{K!}(\partial_{\xi}^Kr_1)(P)(\partial_{x}^Kr_{1})(P))(|\xi|-\lambda)^{-3}.\]
As before, plug this expression into (\ref{heatinv}) and evaluate the contour integrals. We also switch from $r_j$ to $\tilde r_j=-r_j$ for later ease of computation. After all this, we obtain the following expression for $a_2$:
\begin{equation}\label{a2}a_2(P)=(2\pi)^{-n+1}\int_{\mathbb R^{n-1}}b(P,\xi)e^{-|\xi|}\ d\xi,\end{equation}
where
\[b(P,\xi)=\frac{(\tilde r_0(P,\xi))^2}{2}+\tilde r_{-1}(P,\xi)-\frac{i}{2}(\partial_{\xi}\tilde r_1\cdot\partial_{x}\tilde r_0)\]\begin{equation}\label{eqone}-\frac{1}{6}\sum_{\gamma,\epsilon}(\partial_{\xi^\epsilon}\tilde r_1\partial_{\xi^{\gamma}}\tilde r_1\partial_{x^{\gamma}}\partial_{x^{\epsilon}}\tilde r_1)-\frac{1}{2}(\sum_{|K|=2}\frac{1}{K!}(\partial_{\xi}^K\tilde r_1)(P,\xi)(\partial_{x}^K\tilde r_{1})(P,\xi)).\end{equation}

Our strategy is a direct approach: write out $b(P,\xi)$ in terms of the metric and then integrate to get $a_2(P)$. In this calculation, $b(P,\xi)$ splits naturally into two components $b_n(P,\xi)$ and $b_t(P,\xi)$, which we call the \emph{normal} and \emph{tangential} components respectively. To define $b_n$ and $b_t$, first write out the expression for $\hat r_{-1}$ in terms of $\hat r_1$ and $\hat r_0$:
\begin{multline}\label{eqtwo}
\hat r_{-1}=\frac{1}{2|\xi|}\left((\hat r_0)^2-i\sum_{\gamma}(\partial_{\xi^{\gamma}}\hat r_0\partial_{x^{\gamma}}\hat r_1+\partial_{\xi^{\gamma}}\hat r_1\partial_{x^{\gamma}}\hat r_0)-\sum_{|K|=2}\frac{1}{K!}(\partial_{\xi}^K\hat r_1\partial_{x}^K\hat r_1)\right.\\ \left.+\partial_{x^n}\hat r_0+\frac{1}{2}\sum_{\alpha,\beta}g^{\alpha\beta}g_{\alpha\beta,n}\hat r_0\right).\end{multline}
The normal component $b_n$ consists of the terms involving normal derivatives of the metric; as we will see, these are the first term in (\ref{eqone}) along with the restrictions to $x^n=0$ of the first term and the last two terms in (\ref{eqtwo}). The tangential component $b_t$ is the remainder. The expressions for $b_n$ and $b_t$ may be written out and simplified; note in particular that the term in $\hat r_{-1}$ involving a first derivative of the metric, $\partial_{x^{\gamma}}\hat r_1$, is zero at $P$. Writing out the multi-index notation, and noting that the factor of $K!^{-1}$ precisely compensates for the double-counting, we have:
\begin{equation}\label{normalint}b_n(P,\xi)=(\frac{1}{2}+\frac{1}{2|\xi|})\tilde r_0^2(P,\xi)+\frac{1}{2|\xi|}(\partial_{x^n}\hat r_0)(P,\xi)-\frac{1}{2|\xi|}\sum_{\alpha}\lambda_{\alpha}\tilde r_0(P,\xi);\end{equation}
\begin{multline}\label{tangentialint}b_t(P,\xi)=\left[-i(\frac{1}{2}+\frac{1}{2|\xi|})\sum_{\gamma}\partial_{\xi^{\gamma}}\tilde r_1\cdot\partial_{x^{\gamma}}\tilde r_0-\frac{1}{6}\sum_{\gamma,\epsilon}\partial_{\xi^\epsilon}\tilde r_1\cdot\partial_{\xi^{\gamma}}\tilde r_1\cdot\partial_{x^{\gamma}}\partial_{x^{\epsilon}}\tilde r_1\right.\\ \left.
-(\frac{1}{4}+\frac{1}{4|\xi|})\sum_{\gamma,\epsilon}\partial_{\xi^{\gamma}}\partial_{\xi^{\epsilon}}\tilde r_1\cdot \partial_{x^{\gamma}}\partial_{x^{\epsilon}}\tilde r_1\right](P,\xi).\end{multline}

Now we write $b_n$ and $b_t$ in terms of the metric. Recall that when integrating over the tangent space, any term which is odd in $\xi^{\alpha}$ for any particular $\alpha$ will integrate to zero. We therefore define an equivalence relation $\cong$ on functions of $\xi$ by writing $a(\xi)\cong b(\xi)$ iff
\[\int_{\mathbb R^{n-1}}e^{-|\xi|}(a(\xi)-b(\xi))\ d\xi = 0.\] 

\begin{lemma}\label{compute} Let $R_{ij}$ be the Ricci tensor on $M$. Then:
\[b_n(P,\xi)\cong\frac{1}{8}\sum_{\alpha,\beta}\lambda_{\alpha}\lambda_{\beta}(1+\frac{1}{|\xi|})(1-\frac{(\xi^{\alpha})^2}{|\xi|^2})(1-\frac{(\xi^{\beta})^2}{|\xi|^2})+\frac{1}{2|\xi|^5}\sum_{\alpha,\beta}\lambda_{\alpha}\lambda_{\beta}(\xi^{\alpha})^2(\xi^{\beta})^2\]
\[-\frac{1}{8|\xi|^3}\sum_{\alpha}g^{\alpha\alpha,nn}(\xi^{\alpha})^2+\frac{1}{2|\xi|}\sum_{\alpha}\lambda_{\alpha}^2-\frac{1}{8|\xi|}\sum_{\alpha}g_{\alpha\alpha,nn}-\frac{1}{4|\xi|}\sum_{\alpha,\beta}\lambda_{\alpha}\lambda_{\beta}(1-\frac{(\xi^{\beta})^2}{|\xi|^2});\]
\[b_t(P,\xi)\cong(\frac{1}{12|\xi|^3}+\frac{1}{12|\xi|^2})\sum_{\alpha}R_{\alpha\alpha}(\xi^{\alpha})^2.
\]
\end{lemma}
The proof is a direct calculation; we plug in the formulas for the $\tilde r_i$ and simplify. The analysis of $b_t$ uses the Taylor expansion of a Riemannian metric in normal coordinates as well as symmetries of the Riemann curvature tensor. The details of the proof are deferred to section \ref{pcompute}.

Since we have found $b=b_n+b_t$, we proceed to compute $a_2(P)$, which we correspondingly write $a_{2,n}(P)+a_{2,t}(P)$. In the computation, the integrals in section \ref{auxint} are useful. Each integral contains a factor of $V_n$, and since each integral has either $k=0$ or $k=-1$, we can also bring out a factor of $\Gamma(n-2)$ (and then multiply the $k=0$ integrals by $n-2$ to compensate). We obtain:
\[a_{2,t}(P)=\frac{V_n\Gamma(n-2)}{12(2\pi)^{n-1}}R_M,\]
\begin{multline}a_{2,n}(P)=\frac{V_n\Gamma(n-2)}{(2\pi)^{n-1}}\left[\frac{1}{8}\sum_{\alpha,\beta}\lambda_{\alpha}\lambda_{\beta}(n-1)(1-\frac{2}{n-1}+\frac{1+2\delta_{\alpha\beta}}{n^2-1})+\frac{1}{2}\sum_{\alpha,\beta}\lambda_{\alpha}\lambda_{\beta}\frac{1+2\delta_{\alpha\beta}}{n^2-1}\right.\\ \left.-\frac{1}{8(n-1)}\sum_{\alpha}g^{\alpha\alpha,nn}+\frac{1}{2}\sum_{\alpha}\lambda_{\alpha}^2-\frac{1}{8}\sum_{\alpha}g_{\alpha\alpha,nn}-\frac{1}{4}\sum_{\alpha,\beta}\lambda_{\alpha}\lambda_{\beta}(1-\frac{1}{n-1})\right].\end{multline}
Simplifying and using (\ref{secondorderexp}), we find
\begin{multline}a_{2,n}(P)=\frac{V_n\Gamma(n-2)}{8(2\pi)^{n-1}}\left[(n-5+\frac{2}{n-1}+\frac{n+3}{n^2-1})\sum_{\alpha,\beta}\lambda_{\alpha}\lambda_{\beta}\right.\\ \left.+(\frac{2n+6}{n^2-1}+4-\frac{8}{n-1})\sum_{\alpha}\lambda_{\alpha}^2-\frac{n-2}{n-1}\sum_{\alpha}g_{\alpha\alpha,nn}\right].\end{multline}
Finally, recall that
\begin{equation}\label{eigencurv}\sum_{\alpha,\beta}\lambda_{\alpha}\lambda_{\beta}=(n-1)^2H_1^2;\ \ \sum_{\alpha}\lambda_{\alpha}^2=(n-1)^2H_1^2-(n-1)(n-2)H_2,\end{equation}so we may rewrite the eigenvalue sums as combinations of $H_1^2$ and $H_2$. Moreover,
the last term in the expression for $a_{2,n}$ may also be rewritten in terms of mean curvatures:
\begin{lemma}\label{scalcurvform}
\[-\sum_{\alpha}g_{\alpha\alpha,nn}(P)=R_{\Omega}-R_M-2(n-1)^2H_1^2+3(n-1)(n-2)H_2.\]
\end{lemma}
This proof is deferred to section \ref{pscal}.
From Lemma \ref{scalcurvform} and (\ref{eigencurv}), we conclude after some algebraic manipulations that
\begin{multline}a_{2,n}(P)=\frac{V_n\Gamma(n-2)}{8(2\pi)^{n-1}}\left[\frac{(n-1)(n-2)(n^2-n-4)}{n+1}H_1^2\right.\\ \left.-\frac{n(n-3)(n-2)}{n+1}H_2+\frac{n-2}{n-1}(R_{\Omega}-R_M)\right].\end{multline}
Combining this with $a_{2,t}(P)$ yields Theorem \ref{explicit}.
\subsection{Example: heat invariants of balls in $\mathbb{R}^3$ and $\mathbb{R}^4$}
\label{sec:example}
 In this section we verify the formulas for the heat invariants obtained using Theorem \ref{explicit} by calculating them directly for the  
unit balls $\mathbb{B}^3 \subset  \mathbb{R}^3$ and $\mathbb{B}^4 \subset \mathbb{R}^4$. It is well-known that the eigenvalues of the Steklov problem on a unit ball $\mathbb{B}^n \subset \mathbb{R}^n$ are given by a sequence of natural numbers $k=0,1,2,3\dots$, with each eigenvalue $k$ repeated according to its multiplicity
\begin{equation}
\label{ballmult}
m_k=\frac{(2k+n-2)(k+n-3)!}{k! (n-2)!}.
\end{equation}
Note that the eigenfunctions of the Dirichlet-to-Neumann operator $\D$ on $\mathbb{S}^{n-1}$ are spherical harmonics, and the numbers $m_k$ are the multiplicities 
of the Laplace--Beltrami eigenvalues on $\mathbb{S}^{n-1}$  (see, for example, \cite{Po}).  The Steklov heat trace on a ball is given by an explicit formula
\begin{equation}
\label{heat:balls}
\Tr e^{-t\D}= \sum_{k=0}^{\infty} \frac{(2k+n-2)(k+n-3)!}{k! (n-2)!}\, e^{-kt}.
\end{equation}
Let us compute its asymptotics as $t\to 0+$  for $n=3$ and $n=4$.
For $n=3$, the series \eqref{heat:balls} takes the form
$$
\sum_{k=0}^{\infty} (2k+1) e^{-kt}=\frac{1+e^{-t}}{(1-e^{-t})^2}=t^{-2}\left(2+t+\frac{t^2}{3}+O(t^3)\right),
$$
and therefore the corresponding heat invariants are $a_0=2$, $a_1=1$ and $a_2=1/3$.
Taking into account that $V_3=2 \pi$, $H_1=H_2=1$ at each point $x \in \mathbb{S}^2$ and $\Vol(\mathbb{S}^2)=4\pi$,  we obtain precisely the same values for $a_0, a_1$ and $a_2$ from  Theorem \ref{explicit}.

Let  now $n=4$. Then \eqref{heat:balls} takes the form
$$
\sum_{k=0}^{\infty} (k+1)^2 e^{-kt}=e^t\sum_{k=0}^\infty (k+1)^2 e^{-(k+1)t} = \frac{e^{2t} (1+e^t)}{(e^t-1)^3}=t^{-3}\left(2+2t+t^2+O(t^3)\right).
$$
The corresponding heat invariants are $a_0=2$, $a_1=2$, $a_2=1$. These results are in agreement with Theorem \ref{explicit}, as one can easily check, taking into account that  $V_4=4\pi$, $H_1=H_2=1$ for any $x \in \mathbb{S}^3$ and  $\Vol(\mathbb{S}^3)=2\pi^2$.

\section{Proof of Theorem \ref{corollary:main}}

In this section, we adapt an argument of Zelditch in \cite{z} and combine it with an analysis of the first three heat invariants to prove Theorem \ref{corollary:main}. 

\subsection{Zelditch's theorem on multiplicities}\label{51}
\begin{lemma}\label{lemma:simpconn}
Suppose that $\Omega$ is a compact domain in $\mathbb R^3$, with smooth and connected boundary $M$, and with Steklov spectrum equal to that of the ball of radius $\rho$ in $\mathbb R^3$. Then $M$ is in fact simply connected.
\end{lemma}
\begin{proof} The lemma in fact follows from a stronger result, which is an adaptation of a similar result of Zelditch for the Laplacian (Theorem A in \cite{z}).
\begin{proposition}\label{prop:zelditch} Let $U$ be a compact Riemannian manifold of dimension $n$ with smooth boundary $(Y,g)$, and let $0=\lambda_0<\lambda_1<\lambda_2<\ldots$ be the \emph{distinct} eigenvalues of the Dirichlet-to-Neumann operator $\D$, with multiplicities $m_0$, $m_1$, $\ldots$. Suppose there exists $a>0$ such that
\[m_k=ak^{n-2}+\mathcal O(k^{n-3}).\]
Then $(Y,g)$ is Zoll: that is, all geodesics on $Y$ are periodic with a common period.
\end{proposition}

\begin{proof} The proof is closely analogous to the proof in \cite{z}, which is in turn based on work of Ivrii and H\"ormander. Let $T^*Y$ be the cotangent bundle, and let $\Pi^*(y,\eta)$ be the microlocal period function on $T^{*}Y\setminus\{0\}$, equal to the period of the geodesic corresponding to $(y,\eta)$ if it is periodic and equal to $\infty$ if it is not. Assume for contradiction that there is a non-periodic geodesic; then the set $\Gamma_T$ of $(y,\eta)$ for which $\Pi^*(y,\eta)>T$ is a nonempty open cone for each positive $T$. Fix a large positive $T$ to be determined later. As in \cite{z}, let $B$ be a self-adjoint pseudodifferential operator of order zero with microlocal support nontrivial and contained in $\Gamma_T$, and let $b$ be the principal symbol of $B^*B$. Then define
\[N(\lambda,B^*B)=\sum_{\lambda_k\leq\lambda}\phantom{}^{*} \Tr B^*B|_{E_k},\]
where $E_k$ is the $\mathcal D$-eigenspace corresponding to $\lambda_k$, and the sum $\sum^*$ is over distinct eigenvalues. 

There is again a modified Weyl law for $N(\lambda,B^*B)$ given in the proof of Theorem 29.1.5 in H\"ormander \cite{ho}. The difference with \cite{z} is that the subprincipal symbol of $\mathcal D$, $r_0(y,\eta)$, is not necessarily zero. On the other hand, the subprincipal symbol of $B^*B$ is still zero, and the integral of the Poisson bracket $\{b,|\eta|\}$ is still zero, so from \cite{ho} we have:
\begin{equation}\label{modweyl}N(\lambda,B^*B)=(2\pi)^{-(n-1)}\left(\iint_{|\eta|<\lambda}b\ dy\ d\eta+\partial_{\lambda}\iint_{|\eta|<\lambda}r_0(y,\eta)b(y,\eta)\ dy\ d\eta\right)+R(\lambda,B^*B),\end{equation}
where the remainder $R(\lambda,B^*B)$ satisfies the estimate
\begin{equation}\label{remest}\limsup_{\lambda\rightarrow\infty}\lambda^{-(n-2)}\left|R(\lambda,B^*B)\right|\leq\frac{C}{T}\left|\iint_{|\eta|<1}b\ dy\ d\eta\right|.\end{equation}
We write $\bar{b}=\iint_{\eta<1}b\ dy\ d\eta$ and note that since $B$ has nontrivial microlocal support, $\bar{b}>0$.

The first two terms in (\ref{modweyl}) are continuous in $\lambda$, which allows us to write
\[\Tr B^*B|_{E_k}=\lim_{\epsilon\rightarrow 0}(R(\lambda_k+\epsilon,B^*B)-R(\lambda_k-\epsilon,B^*B)).\] Hence for $\lambda_k$ greater than some $\lambda_0(T)$ (it may depend on $T$), we have
\[\Tr B^*B|_{E_k}\leq\frac{2(C+1)}{T}\bar b\lambda_k^{n-2}.\] Summing gives
\begin{equation}\label{estim}N(\lambda,B^*B)\leq\frac{2(C+1)}{T}\bar b\sum_{\lambda_k\leq\lambda}\phantom{}^{*}\lambda_k^{n-2}+\mathcal O_T(1).\end{equation}

Next, use the multiplicity assumption, precisely as in \cite{z}, to rewrite (\ref{estim}); the proof is identical to that in \cite{z}, albeit with $\lambda_k$ here replacing $\sqrt{\lambda_k}$ in \cite{z}, so we omit it. We conclude that
\begin{equation}\label{estimtwo}N(\lambda,B^*B)\leq\frac{D}{T}\bar b\lambda^{n-1}+\mathcal O(\lambda^{n-2})+\mathcal O_T(1),\end{equation}
where $D$ is a constant depending only on $n$ and the volume of $Y$. Note that the first term in (\ref{modweyl}) is exactly $(2\pi)^{-(n-1)}\bar{b}\lambda^{n-1}$, and all other terms are $\mathcal O(\lambda^{n-2})$. Thus fixing $T>(2\pi)^{n-1} D$ and then letting $\lambda\rightarrow\infty$ in (\ref{estimtwo}) contradicts (\ref{modweyl}).

We have now shown that all geodesics on $Y$ are periodic; as in \cite{z}, we now apply theorem (0.40) in \cite{be} to conclude that all the closed geodesics must have a common period, and therefore that $(Y,g)$ is Zoll. \end{proof}

\begin{remark}  Suppose, on the other hand, that $U$ is a manifold with boundary $Y$ having the property that the periodic geodesics on $Y$ form a set of measure zero in $T^*Y$. Then, as follows from the results of \cite{DG} (see also \cite[Chapter 29]{ho}),  there is a two-term Weyl law for the Steklov eigenvalues.  We may write it as 
\begin{equation}
\label{wl2}
\#(\sigma_k\leq\sigma)=\frac{a_0}{(n-1)!}\sigma^{n-1}+\frac{a_1}{(n-2)!}\sigma^{n-2}+o(\sigma^{n-2}),
\end{equation}
where $a_0$ and $a_1$ are the heat invariants.  Here we have used the Laplace transform to relate the heat trace and the counting function.
\end{remark}

Now we finish the proof of Lemma \ref{lemma:simpconn}. We see from (\ref{ballmult}) that $\Omega$ satisfies the hypothesis of Proposition \ref{prop:zelditch}. Therefore $M$ must be Zoll; in fact, the same argument holds for any domain in $\mathbb R^n$ which is Steklov-isospectral to a ball. Since all connected Zoll surfaces embedded in $\mathbb R^3$ are topological spheres \cite{be}, we conclude that  $M$ is simply connected. \end{proof}

On the other hand, since there are large families of Zoll surfaces \cite{gu}, we cannot immediately conclude that $M$ is a sphere.

\subsection{Application of heat invariants}\label{52} We now use the heat invariants we have computed to finish the proof of Theorem \ref{corollary:main}. Let $\chi(M)$ be the Euler characteristic. By (\ref{3d}), the Gauss-Bonnet theorem, and Lemma \ref{lemma:simpconn}, we know that the second Steklov heat invariant

\begin{equation}
\label{a2eq}
a_2=\frac{1}{16\pi}\int_M H_1^2+\frac{1}{24}\chi(M)=\frac{1}{16\pi}\int_M H_1^2+\frac{1}{12}.
\end{equation}
Therefore $\int_M H_1^2=\int_{S_\rho} H_1^2$. On the other hand, we already know from the first two heat invariants that $\Vol(M)$ and $\int_M H_1$ are Steklov spectral invariants, so $\Vol(M)=\Vol(S_\rho)$ and $\int_M H_1=\int_{S_\rho} H_1$. Therefore
\begin{equation}\label{punchline}\sqrt{\Vol(M)}\left(\int_M H_1^2\right)^{1/2}-\left|\int_M H_1\right|=\sqrt{\Vol(S_\rho)}\left(\int_{S_\rho} H_1^2\right)^{1/2}-\left|\int_{S_\rho} H_1\right|=0.\end{equation}
By the Cauchy-Schwarz inequality, $H_1$ must be constant on $M$. However, the only embedded compact surfaces of constant mean curvature in $\mathbb R^3$ are round spheres \cite{Al}, so we conclude that $M$ is itself a sphere of radius $\rho$ and therefore $\Omega$ is isometric to $B_{\rho}$. This completes the proof.

\section{Proofs of auxiliary lemmas}

The following two proofs are given by lengthy but straightforward computations.

\subsection{Proof of Lemma \ref{compute}} 
\label{pcompute}
To analyze the normal integrand (\ref{normalint}), first recall the expression for $r_0(P,\xi)$ given in (\ref{rzeroatp}). Plugging in this expression takes care of the first and last terms in (\ref{normalint}); it remains only to analyze the middle term $\frac{1}{2|\xi|}\partial_{x^n}\hat r_0$. Examining the expression (\ref{writtenout}) for $\hat r_0$, we notice that the first and third terms (the imaginary part of $\hat r_0$) are odd in $\xi$, and hence the same is true after applying a normal derivative and dividing by $2|\xi|$. Therefore
\[\frac{1}{2|\xi|}\partial_{x^n}\hat r_0\cong\frac{1}{2|\xi|}\partial_{x^n}\left(-\frac{1}{4|\xi|^2}\sum_{\alpha,\beta}g^{\alpha\beta,n}\xi^{\alpha}\xi^{\beta}
-\frac{1}{4}\sum_{\alpha,\beta}g^{\alpha\beta}g_{\alpha\beta,n}\right)\]
\[=-\frac{1}{8|\xi|^3}\sum_{\alpha,\beta}g^{\alpha\beta,nn}\xi^{\alpha}\xi^{\beta}-\frac{1}{8|\xi|}\sum_{\alpha,\beta}g^{\alpha\beta,n}g_{\alpha\beta,n}-\frac{1}{8|\xi|}\sum_{\alpha,\beta}g^{\alpha\beta}g_{\alpha\beta,nn}.\]
\[\cong-\frac{1}{8|\xi|^3}\sum_{\alpha}g^{\alpha\alpha,nn}(\xi^{\alpha})^2+\frac{1}{2|\xi|}\sum_{\alpha}\lambda_{\alpha}^2-\frac{1}{8|\xi|}\sum_{\alpha}g_{\alpha\alpha,nn}.\]
In the final step of this calculation, we used the fact that $|\xi|^{-3}\xi^{\alpha}\xi^{\beta}\cong 0$ whenever $\alpha\neq\beta$. Combining this calculation with the rest of the integrand completes the proof of the normal portion of Lemma \ref{compute}.

\subsubsection*{Tangential integrand: initial computations} Now consider the tangential portion. The first component of $b_t$ is $-i(\frac{1}{2}+\frac{1}{2|\xi|})\sum_{\gamma}\partial_{\xi^{\gamma}}\hat r_1\partial_{x^{\gamma}}\hat r_0$. Notice that
\[(\partial_{\xi^{\gamma}}\hat r_1)(P)=-\frac{1}{2|\xi|}g^{\alpha\beta}(\delta^{\gamma\beta}\xi^{\alpha}+\delta^{\gamma\alpha}\xi^{\beta})=-\frac{\xi^{\gamma}}{|\xi|},\]
which is odd in $\xi$. Each term in the real part of $\partial_{x^{\gamma}}\hat r_0$ will be a power of $|\xi|$ times a polynomial of even degree in $\xi$, so after multiplying by $\xi^{\gamma}/|\xi|$, it will be equivalent to zero. We therefore need to consider only the imaginary part of $\partial_{x^{\gamma}}\hat r_0$, and hence only the imaginary part of $\hat r_0$ itself.
However, from (\ref{writtenout}), each term in the imaginary part of $\hat r_0$ is multiplied by a first derivative of the metric or the log of the volume element, which is zero at $P$. Therefore, the $\partial_{x^{\gamma}}$ derivative must always hit that term, or else the result is zero at $P$. Relabeling the dummy variable $\gamma$ in (\ref{writtenout}) as $\epsilon$, to avoid confusion with $\partial_{x^{\gamma}}$, we conclude that
\begin{multline}\partial_{x^{\gamma}}\hat r_0(P)=\left(-\frac{i}{8|\xi|^3}\sum_{\epsilon}\left[(\sum_{\alpha,\beta}g^{\alpha\beta}(\delta^{\epsilon\alpha}\xi^{\beta}+\delta^{\epsilon\beta}\xi^{\alpha}))(\sum_{\alpha,\beta}g^{\alpha\beta,\epsilon\gamma}\xi^{\alpha}\xi^{\beta})\right]\right.\\ \left.+\frac{i}{2|\xi|}\sum_{\alpha,\beta}(\frac{1}{2}g^{\alpha\beta}\frac{\partial^2}{\partial_{x^{\gamma}}\partial_{x^\alpha}}\log\delta+g^{\alpha\beta,\alpha\gamma})\xi^{\beta}\right)(P).\end{multline}
Now plug in the values of the metric at $P$, and then multiply by $i\xi^{\gamma}(\frac{1}{2}+\frac{1}{2|\xi|})$. As before, we also note that any term of the form $|\xi|^{-k}\xi^{\alpha}\xi^{\beta}$ is equivalent to zero whenever $\alpha\neq\beta$. We get that the first component of $b_t(P,\xi)$ is equivalent to:
\[(\frac{1}{8|\xi|^5}+\frac{1}{8|\xi|^4})\sum_{\alpha,\beta,\gamma,\epsilon}g^{\alpha\beta,\gamma\epsilon}(P)\xi^{\alpha}\xi^{\beta}\xi^{\gamma}\xi^{\epsilon}-(\frac{1}{8|\xi|^3}+\frac{1}{8|\xi|^2})\sum_{\alpha}(\frac{\partial^2}{\partial_{(x^{\alpha})^2}}\log\delta)(P)(\xi^{\alpha})^2\]\begin{equation}\label{termthree}-(\frac{1}{4|\xi|^3}+\frac{1}{4|\xi|^2})\sum_{\alpha,\beta}g^{\alpha\beta,\alpha\beta}(P)(\xi^{\beta})^2.\end{equation}

For the remaining terms, we first compute second derivatives of $\tilde r_1$ in both the $x$ variables and the $\xi$ variables:
\[(\partial_{\xi^{\gamma}}\partial_{\xi^{\epsilon}}\tilde r_1)(P)=(\frac{\partial_{\xi^{\gamma}}q_2\partial_{\xi^{\epsilon}}q_2}{4|\xi|^3}-\frac{\partial_{\xi^{\gamma}}\partial_{\xi^{\epsilon}}q_2}{2|\xi|})(P)=\frac{\xi^{\gamma}\xi^{\epsilon}}{|\xi|^3}-\frac{\delta^{\gamma\epsilon}}{|\xi|};\]
\[(\partial_{x^{\gamma}}\partial_{x^{\epsilon}}\tilde r_1)(P)=(\frac{\partial_{x^{\gamma}}q_2\partial_{x^{\epsilon}}q_2}{4|\xi|^3}-\frac{1}{2|\xi|}\partial_{x^{\gamma}}\partial_{x^{\epsilon}}q_2)(P)=-\frac{1}{2|\xi|}\sum_{\alpha,\beta}g^{\alpha\beta,\gamma\epsilon}(P)\xi^{\alpha}\xi^{\beta}.\]
From (\ref{tangentialint}), the remaining terms in $b_t(P,\xi)$ are hence equivalent to
\begin{equation}\label{termtangtwo}(\frac{1}{12|\xi|^3}+\frac{1}{8|\xi|^4}+\frac{1}{8|\xi|^5})\sum_{\alpha,\beta,\gamma,\epsilon}g^{\alpha\beta,\gamma\epsilon}(P)\xi^{\alpha}\xi^{\beta}\xi^{\gamma}\xi^{\epsilon}
-(\frac{1}{8|\xi|^2}+\frac{1}{8|\xi|^3})\sum_{\alpha,\beta,\gamma}g^{\alpha\beta,\gamma\gamma}(P)\xi^{\alpha}\xi^{\beta}.\end{equation}

\subsubsection*{Curvatures} We now use the well-known Taylor expansion of the metric in Riemannian normal coordinates (see \cite{v}, for example) to relate the second derivatives of the metric and the volume element to the intrinsic curvatures of the boundary $M$. Let $R_{ijkl}$ be the Riemann curvature tensor of the boundary $M$ at the point $P$; then we have
\[g_{\alpha\beta}(x)=\delta_{\alpha\beta}-\frac{1}{3}R_{\alpha\mu\beta\nu}x^{\mu}x^{\nu}+\mathcal O(|x|^3).\]
An easy inverse argument using Taylor series gives:
\[g^{\alpha\beta}(x)=\delta_{\alpha\beta}+\frac{1}{3}R_{\alpha\mu\beta\nu}x^{\mu}x^{\nu}+\mathcal O(|x|^3),\]
and hence
\begin{equation}\label{helpfulone}g^{\alpha\beta,\gamma\epsilon}(P)=\frac{1}{3}(R_{\alpha\gamma\beta\epsilon}+R_{\alpha\epsilon\beta\gamma}).\end{equation}
As for the volume element, we have from \cite{v}:
\[\delta(x)=1-\frac{1}{3}R_{\mu\nu}x^{\mu}x^\nu+\mathcal O(|x|^3),\]
so by elementary Taylor series arguments
\[\log\delta(x)=-\frac{1}{3}R_{\mu\nu}x^{\mu}x^\nu+\mathcal O(|x|^3),\]
and therefore
\begin{equation}\label{helpfulthree}(\partial_{\xi^{\gamma}}\partial_{\xi^{\epsilon}}\log\delta)(P)=-\frac{1}{3}
(R_{\gamma\epsilon}+R_{\epsilon\gamma})=-\frac{2}{3}R_{\gamma\epsilon}.\end{equation}

Combining these observations with (\ref{termthree}) and (\ref{termtangtwo}), and also removing the $\alpha\neq\beta$ part of the last term of (\ref{termtangtwo}) (which is equivalent to zero), we get:
\[b_t(P,\xi)\cong(\frac{1}{36|\xi|^3}+\frac{1}{12|\xi|^4}+\frac{1}{12|\xi|^5})\sum_{\alpha,\beta,\gamma,\epsilon}(R_{\alpha\gamma\beta\epsilon}+R_{\alpha\epsilon\beta\gamma})\xi^{\alpha}\xi^{\beta}\xi^{\gamma}\xi^{\epsilon}\]\[+(\frac{1}{12|\xi|^3}+\frac{1}{12|\xi|^2})\sum_{\alpha}R_{\alpha\alpha}(\xi^{\alpha})^2-(\frac{1}{12|\xi|^3}+\frac{1}{12|\xi|^2})\sum_{\alpha,\beta}(R_{\alpha\alpha\beta\beta}+R_{\alpha\beta\beta\alpha})(\xi^{\beta})^2\]\begin{equation}\label{simplifythis}
-(\frac{1}{12|\xi|^2}+\frac{1}{12|\xi|^3})\sum_{\alpha,\gamma}R_{\alpha\gamma\alpha\gamma}(\xi^{\alpha})^2.\end{equation}

\subsubsection*{Final simplification} To simplify (\ref{simplifythis}) further, we use the symmetries of the curvature tensor:
\[R_{ijkl}=-R_{jikl}=-R_{ijlk}=R_{klij}.\]
From these symmetries, we immediately conclude that $R_{\alpha\alpha\beta\beta}=0$ for all $\alpha$ and $\beta$, and also that $R_{\alpha\beta\beta\alpha}=-R_{\alpha\beta\alpha\beta}$. Hence the third and fourth terms of (\ref{simplifythis}) cancel. Moreover:
\begin{proposition} For any $k\geq -3-n$ (so that the integral over the tangent space makes sense), \[|\xi|^k\sum_{\alpha,\beta,\gamma,\epsilon}(R_{\alpha\gamma\beta\epsilon}+R_{\alpha\epsilon\beta\gamma})\xi^{\alpha}\xi^{\beta}\xi^{\gamma}\xi^{\epsilon}\cong 0.\]
\end{proposition}
\begin{proof}
Integration of a multiple of $\xi^{\alpha}\xi^{\beta}\xi^{\gamma}\xi^{\epsilon}$ against an even function of $\xi$ vanishes unless the four indices pair off into groups of two. On the other hand, if all four are the same, the curvature coefficients are zero. We may therefore consider each possible pairing separately without worrying about double counting; for example, if $\alpha=\beta$ and $\gamma=\epsilon$, the sum becomes
\[2\sum_{\alpha,\gamma}R_{\alpha\gamma\alpha\gamma}(\xi^{\alpha})^2(\xi^{\gamma})^2.\]
Relabeling and adding up the possibilities, using the symmetries of the curvature tensor, the first sum becomes
\[\sum_{\alpha,\beta}(2R_{\alpha\beta\alpha\beta}+R_{\alpha\beta\beta\alpha}+R_{\alpha\beta\beta\alpha})(\xi^{\alpha})^2(\xi^{\beta})^2=0,\]
which completes the proof of the proposition.
\end{proof}

As a consequence of the proposition and the preceding remarks, the first term in (\ref{simplifythis}) is equivalent to zero, and we are left with \[b_t(P,\xi)\cong (\frac{1}{12|\xi|^3}+\frac{1}{12|\xi|^2})\sum_{\alpha}R_{\alpha\alpha}(\xi^{\alpha})^2.\] This completes the proof of Lemma \ref{compute}.

\subsection{Proof of Lemma \ref{scalcurvform}} 
\label{pscal}
The proof is by computation of the scalar curvature $R_{\Omega}$. At the point $P$, using Einstein notation, we have
\begin{equation}\label{scalcurv}R_{\Omega}= \Gamma^c_{aa,c}-\Gamma^c_{ac,a}+\Gamma^{d}_{aa}\Gamma^c_{cd}-\Gamma^d_{ac}\Gamma^c_{ad}.\end{equation}
Moreover, if the sums are taken only over the indices from 1 to $n-1$, rather than from 1 to $n$, we obtain the scalar curvature of the boundary. So $R_{\Omega}$ equals $R_M$ plus the terms in (\ref{scalcurv}) where at least one of $a$, $c$, or $d$ is $n$. On the other hand, the Christoffel symbols are given by
\[\Gamma_{ij}^k=\frac{1}{2}\sum_m(g_{jm,i}+g_{mi,j}-g_{ij,,m})g^{km}.\]
Recall that $g_{\alpha n}$ and $g^{\alpha n}$ are identically zero in a neighborhood of $P\subset\Omega$ for any $\alpha\leq n-1$, and $g_{nn}$ and $g^{nn}$ are identically 1. As a consequence, whenever two or more of $i$, $j$, and $k$ are equal to $n$, the Christoffel symbol $\Gamma_{ij}^k$ is zero in a neighborhood of $P$.

We first analyze the first two terms in (\ref{scalcurv}) where at least one index is $n$. Using the observation above on the vanishing of the Christoffel symbols, we obtain that at $P$,
\[\sum_{a=1}^{n-1}\Gamma_{aa,n}^n-\sum_{c=1}^{n-1}\Gamma^{c}_{nc,n}=\sum_{\alpha=1}^{n-1}\Gamma_{\alpha\alpha,n}^n-\sum_{\alpha=1}^{n-1}\Gamma^{\alpha}_{n\alpha,n}\]\[=\frac{1}{2}\sum_{\alpha=1}^{n-1}\partial_{x^n}(\sum_{m}(g_{\alpha m,\alpha}+g_{m \alpha,\alpha}-g_{\alpha\alpha, m})g^{nm})+\frac{1}{2}\sum_{\gamma=1}^{n-1}\partial_{x^n}(\sum_m(g_{\alpha m, n}+g_{mn,\alpha}-g_{n\alpha,n})g^{\alpha m})\]
\[=\frac{1}{2}\sum_{\alpha=1}^{n-1}(-g_{\alpha\alpha,nn})-\frac{1}{2}\sum_{\alpha=1}^{n-1}\partial_{x^n}(\sum_{m=1}^{n-1}g_{\alpha m,n}g^{\alpha m})=-\sum_{\alpha=1}^{n-1}g_{\alpha\alpha,nn}+2\sum_{\alpha=1}^{n-1}\lambda_{\alpha}^2.\]

Now consider the final two terms of (\ref{scalcurv}). Since there are no derivatives of the Christoffel symbols involved in these terms, we can plug in the metric at $P$ and write
\[\Gamma_{ij}^k(P)=\frac{1}{2}(g_{jk,i}+g_{ki,j}-g_{ij,k})(P).\] If more than one of $i$, $j$, or $k$ is $n$, then $\Gamma_{ij}^k$ vanishes; on the other hand, for $\alpha,\beta\neq n$, we compute
\[\Gamma_{\alpha\beta}^n(P)=\lambda_{\alpha}\delta_{\alpha\beta}; \Gamma_{\alpha n}^{\beta}=\Gamma_{\beta n}^{\alpha}=-\lambda_{\alpha}\delta_{\alpha\beta}.\]

If more than one of $a$, $c$, or $d$ is equal to $n$, the final two terms of (\ref{scalcurv}) vanish; if none are $n$, then we get part of the scalar curvature $R_M$. If one is $n$, then there are three possibilities:
\begin{itemize}
\item Suppose $a=n$, $c\neq n$, $d\neq n$. Then the third term of (\ref{scalcurv}) is zero; the last term is
\[\sum_{c,d=1}^{n-1}-(-\lambda_{c}\delta_{cd})^2=-\sum_{\alpha=1}^{n-1}\lambda_{\alpha}^2.\]
\item Suppose $a\neq n$, $c=n$, $d\neq n$; the third term is again zero, and the last term is
\[\sum_{a,d=1}^{n-1}-(-\lambda_{a}\delta_{ad})(\lambda_{a}\delta_{ad})=\sum_{\alpha=1}^{n-1}\lambda_{\alpha}^2.\]
\item Finally, suppose $a\neq n$, $c\neq n$, $d=n$. Then we get:
\[\sum_{a,c=1}^{n-1}\lambda_a(-\lambda_c)+\sum_{a,c}-(-\lambda_{a}\delta_{ac})(\lambda_a\delta_{ac})=-(n-1)^2H_1^2+\sum_{\alpha}\lambda_{\alpha}^2.\]
\end{itemize}

Combining these three computations with the first two terms of (\ref{scalcurv}) yields
\[R_{\Omega}=R_M+3\sum_{\alpha}\lambda_{\alpha}^2-\sum_{\alpha}g_{\alpha\alpha,nn}-(n-1)^2H_1^2.\]
Rearranging and using (\ref{eigencurv}) completes the proof.

\subsection{Auxiliary Integrals}
\label{auxint}
\begin{lemma} Let $\alpha$ and $\beta$ be any two distinct integers between $1$ and $n-1$. Moreover, let $V_n=\Vol(\mathbb{S}^{n-2})$. For any real $k$ such that the following integrals converge, we have
\begin{equation}\label{intone}\int_{\mathbb R^{n-1}}e^{-|\xi|}|\xi|^k\ d\xi=V_n\Gamma(k+n-1);\end{equation}

\begin{equation}\label{inttwo}\int_{\mathbb R^{n-1}}e^{-|\xi|}|\xi|^{k-2}(\xi^{\alpha})^2\ d\xi=\frac{V_n}{n-1}\Gamma(k+n-1);\end{equation}

\begin{equation}\label{intthree}\int_{\mathbb R^{n-1}}e^{-|\xi|}|\xi|^{k-4}(\xi^{\alpha})^4 d\xi=\frac{3V_n}{n^2-1}\Gamma(k+n-1);\end{equation}

\begin{equation}\label{intfour}\int_{\mathbb R^{n-1}}e^{-|\xi|}|\xi|^{k-4}(\xi^{\alpha})^2(\xi^{\beta})^2 d\xi=\frac{V_n}{n^2-1}\Gamma(k+n-1).\end{equation}
\end{lemma}

The first integral may be evaluated directly, and by summing over $\alpha$ and using symmetry we see that the second integral is the first integral times $1/(n-1)$. 

For the third and fourth integrals, one uses $n$-dimensional spherical coordinates to reduce the problem to trigonometric integrals. We first perform the third integral: by symmetry we may let $\alpha=1$. Use the coordinates $(r,\theta_i)$ given by
\[|\xi|=r,\ \xi_1=|\xi|\cos\theta_1,\ \xi_k=|\xi|\cos\theta_k\prod_{i=1}^{k-1}\sin\theta_i\ (\textrm{for } 2\leq k\leq n-2),\ \xi_{n-1}=|\xi|\prod_{i=1}^{n-2}\sin\theta_i.\]
Here $\theta_{n-2}\in[0,2\pi)$ and $\theta_k\in[0,\pi)$ for all other $k$. The integral in $r$ may be performed explicitly and yields $\Gamma(k+n-1)$; the rest of the integral (\ref{intthree}) is
\[\int_{S^{n-2}}\cos^4\theta_1\sin^{n-3}\theta_1\sin^{n-4}\theta_2\ldots\sin\theta_{n-3}\ d\theta_1\ldots d\theta_{n-2}.\]
However, since
\[V_n=\int_{S^{n-2}}\sin^{n-3}\theta_1\sin^{n-4}\theta_2\ldots\sin\theta_{n-3}\ d\theta_1\ldots d\theta_{n-2},\]
we see that (\ref{intthree}) is $\Gamma(k+n-1)$ times $V_n$ times
\[L_n:=\frac{\int_0^{\pi}\cos^4\theta_1\sin^{n-3}\theta_1\ d\theta_1}{\int_0^{\pi}\sin^{n-3}\theta_1\ d\theta_1}.\]
We then use \cite[formula 3.621]{gr} and some identities for the Gamma function to show that $L_n=\frac{3}{n^2-1}$.

Finally we see by writing out $|\xi|^4=(\xi_1^2+\ldots+\xi_{n-1}^2)^2$ and using symmetry that $(n-1)$ times (\ref{intthree}) plus $(n-1)(n-2)$ times (\ref{intfour}) equals (\ref{intone}), which enables us to determine (\ref{intfour}) as well.

\end{document}